\documentclass[11pt]{amsart}
\usepackage[margin=1in]{geometry}

\usepackage{hyperref}

\usepackage{amssymb}
\usepackage{amsmath}
\usepackage{amsfonts}
\usepackage{amsthm}

\usepackage{tikz}
\usetikzlibrary{positioning}

\theoremstyle{plain}
\newtheorem{theorem}{Theorem}
\newtheorem{lemma}[theorem]{Lemma}
\newtheorem{corollary}[theorem]{Corollary}

\theoremstyle{definition}
\newtheorem{definition}[theorem]{Definition}

\theoremstyle{remark}
\newtheorem{remark}[theorem]{Remark}

\title{Note on the negative base xor sequence $n\oplus_{-b}(-n)$}
\author{H.\ A.\ Verrill}

\begin{document}

\maketitle

\begin{abstract}
  We prove a relationship between certain integer expressions involving
  operators similar to the binary exclusive or.  This gives a proof and
generalization of a result conjectured about sequence 
A178729 in Sloane's Online Encyclopedia of Integer Sequences. 
We use transducers to prove this result.
\end{abstract}

\section{Introduction}

It was conjectured \cite[A178729]{oeis} that for non-negative
integers $n$, 
\begin{equation}n\oplus (3n) = n\oplus_{-2} (-n),
\label{equ:originalconjecture}
\end{equation}
where on the left side, $\oplus$ is the bitwise exclusive or operator,
defined
by 
\begin{equation}
\left(\sum_{i\ge 0} a_i 2^i\right)
\oplus
\left(\sum_{i\ge 0} b_i 2^i\right) = 
\sum_{i\ge 0} \bigl((a_i + b_i)\mod2\bigr) 2^i,
\label{eqn:defoplus2}
\end{equation}
where $a_i,b_i\in\{0,1\}$, and
mod is defined by
$$
a\mod b:=r\text{ where }0\le r<b \text{ and }b|(a-r).
$$
On the right side of (\ref{equ:originalconjecture}),
$\oplus_{-2}$ is defined by
the formula
\begin{equation}
\left(\sum_{i\ge 0} a_i (-2)^i\right)
\oplus_{-2}
\left(\sum_{i\ge 0} b_i (-2)^i\right) = 
\sum_{i\ge 0} \bigl((a_i + b_i)\mod2\bigr) 2^i,
\label{eqn:defoplus-2}
\end{equation}
with $a_i,b_i\in\{0,1\}$.
The arguments 
$\alpha:=\sum a_i (\epsilon2)^i$ and $\beta:=\sum b_i(\epsilon2)^i$
of (\ref{eqn:defoplus2}), where $\epsilon=1$,
must be non-negative integers,
whereas
$\alpha\oplus_{-b}\beta$ 
 (\ref{eqn:defoplus-2}), where $\epsilon=-1$,
is defined for all integers.

\begin{remark} 
  On the right-hand side of (\ref{eqn:defoplus-2}), we have
a sum of powers of $2$, whereas on the left, the sums are of powers
of $-2$.
\end{remark}

The first few terms of (\ref{equ:originalconjecture}) 
\cite[A178729]{oeis}, starting with $n=1$, are given by
$$2, 4, 10, 8, 10, 20, 18, 16, 18, 20.$$

A table of the coefficients of the expansions of some of the 
first few terms in (\ref{equ:originalconjecture})
are as follows, where $n_b$ means the coefficients 
of the base $b$ expansion of $n$.  Leading zeros are added
so that expressions are the same length, for easy comparison.
Digits correspond from right to left to increasing powers of $2$ or
$-2$, e.g., $110110$ in base $-2$ represents
$-2^5+2^4 + 2^2 - 2=-14$.
$$
\begin{array}{lllllll}
  \hline
n & n_2 & (3n)_2 & n_{-2} & (-n)_{-2} & ((3n)\oplus_2n)_2 & ((3n)\oplus_2n) \\
  \hline
10 & 1010 & 011110 & 11110 &001010 & 010100 & 20\\  
11 & 1011 & 100001 & 11111 &110101& 101010 &42\\
12 & 1100 & 100100 & 11100 &110100& 101000 &40\\
13 & 1101 & 100111 & 11101 &110111& 101010 &42\\
14 & 1110 & 101010 & 10010 &110110& 100100 &36\\
\hline
\end{array}
$$

\subsection{Generalized Result}
\label{subsec:gen}
We prove a more general version of 
(\ref{equ:originalconjecture}). 
First we introduce some notation.

\begin{definition}
\label{def:ominus}
We denote the set of
non-negative integers by $\mathbb N$, and the set of all integers
by $\mathbb Z$.
For a positive integer $b$, we define
the binary operators $\oplus_{-b}$ and $\ominus_b$ as follows.
The results of both operations are expansions in terms of powers
of $b$.  Thus these operators are functions
$\oplus_{-b}:\mathbb Z\times \mathbb Z\rightarrow \mathbb N$ and
$\ominus_{b}:\mathbb N\times \mathbb N\rightarrow \mathbb N$.

Let $\alpha$ and $\beta$ be any two non-negative integers.
They can be written uniquely 
as
$\alpha=\sum_{i\ge0}a_ib^i$
and $\beta=\sum_{i\ge0}b_ib^i$, with
$a_i,b_i\in\{0,\dots,b-1\}$, and only finitely many non-zero
terms.  We define
\begin{equation}
\left(\sum_{i\ge 0} a_i b^i\right)
\ominus_b
\left(\sum_{i\ge 0} b_i b^i\right) = 
\sum_{i\ge 0} \bigl((a_i - b_i)\mod b\bigr) b^i.
\label{eqn:minusb}
\end{equation}
Now let $\alpha$ and $\beta$ be any two integers.
They can be written uniquely 
as $\alpha=\sum_{i\ge0}a_i(-b)^i$
and $\beta=\sum_{i\ge0}b_i(-b)^i$, with
$a_i,b_i\in\{0,\dots,b-1\}$.  We define
\begin{equation}
\left(\sum_{i\ge 0} a_i (-b)^i\right)
\oplus_{-b}
\left(\sum_{i\ge 0} b_i (-b)^i\right) = 
\sum_{i\ge 0} ((a_i + b_i)\mod b) b^i.
\label{eqn:oplus-b}
\end{equation}
For $\alpha=\sum_{i\ge0}a_ib^i\in\mathbb N$ with $b_i\in \{0,\dots,b-1\}$,
we define $\overline{\overline{\alpha}}$ by
\begin{equation}
\overline{\overline{\alpha}}:=
  \sum_{i\ge0}\min(a_i,1)b^i.
\end{equation}

\end{definition}

\begin{remark}
Theorem~\ref{thm:main},
states that for all positive integers $n$ and $b$,
\begin{equation}
\overline{\overline{((b+1)n)\ominus_b n}} = n\oplus_{-b} (-n).
\label{equ:generalresult}
\end{equation}

  In the case that $b=2$,
  (\ref{equ:generalresult}) is exactly the same as
  (\ref{equ:originalconjecture}).
\end{remark}
\section{Finite State Machines}

We prove 
(\ref{equ:originalconjecture}), and a generalization,
using a transducer,
also known as a Mealy machine,
which is a type of deterministic finite state automata with output.
We recall some definitions from \cite[\S4.3]{AutomaticSequences}
for convenience.
\begin{definition}
  A  transducer is the following sextuple
  $$T=(Q,\Sigma,\delta,q_0,\Delta,\lambda),$$
  where

  \begin{tabular}{ll}
    $Q$ & is a finite set, known as the set of states,\\
    $\Sigma$ & is a finite set, known as the alphabet,\\
    $\delta$ & is a transition function,
    $\delta:Q\times\Sigma\rightarrow Q$,\\
    $q_0$ & is an element of $Q$, called the {\it initial state},\\
    $\Delta$ & is the output alphabet (for us, $\Delta=\Sigma$),\\
    $\lambda$ & is the output function,
    $\lambda:Q\times\Sigma\rightarrow\Delta^*$.    
  \end{tabular}
\end{definition}

Here, $\Delta^*$ is the set of finite
words over the alphabet $\Delta$.  For the remainder of the paper we set
$$\Delta=\Sigma:=\{0,1,\dots,b-1\}.$$
  A transducer describes the process of constructing an
  output string from a given input string.
  Given a string $a_ka_{k-1}\dots a_1a_0\in\Sigma^*$, we compute the output string
  to be the concatenation
  \begin{equation}
    \label{eqn:defofmealy}
  b_kb_{k-1}\dots b_1b_0,
  \end{equation}
  where $b_0=\lambda(q_0,a_0)$,
  $b_i=\lambda(q_i,a_i)$, for $i=1,\dots,k$, and the
  sequence of states is defined by $q_i=\sigma(q_{i-1},a_i)$ for
  $i=1,\dots,k$.
  We work with strings read from right to left, since this is the usual
  convention used for the Arabic place value number system.
  We use transducers to describe
  functions such as $n\mapsto ((b+1)n)\ominus_b n$.
  The input is given as a base $b$ expansion of $n$, that is,
  $a_k\dots a_0\in\Sigma^*$, where $n=\sum_{i= 0}^ka_ib^i$.
  The output is also an element of $\Sigma^*$.

  A  transducer is usually represented by a diagram
  consisting of circles labeled with
  states, and
  labeled arrows between
  states, corresponding to $\sigma$ and $\lambda$.
  For example, suppose that $\delta(q_1,\sigma_1)=q_2$, and
  $\lambda(q_1,\sigma_1)=\sigma_2$, with
  $q_1,q_2\in Q,\sigma_1,\sigma_2\in\Sigma$,
  then this is represented by the following subdiagram of the diagram
  of the machine:
  $$
  
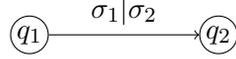
\begin{tikzpicture}[node distance=2 cm]
    \node[circle,draw,inner sep=0pt,minimum size=0.5cm](a){$q_1$};
\node[circle,draw,inner sep=0pt,minimum size=0.5cm, right=of a](b){$q_2$};
\draw[->] (a) -- node[above]{$\sigma_1 | \sigma_2$} (b); 
  \end{tikzpicture}
  $$
  The diagram for a transducer $T$ encapsulates all the information
  required to define $T$.

\section{Proof of the Result}

In the following, $n$ is a non-negative integer, and
$b$ is a positive integer.
Let $Q=\{$``even 0'', ``odd 0'', ``even 1'', ``odd 1''$\}$,
 $\Sigma=\{0,1,\dots,b-1\}$, $q_0=$ ``even 0'',
 and $\delta, \lambda$ are as indicated in
 the following diagram.
 Then the transducer $T=(Q,\Sigma,\delta,q_0,\Delta,\lambda)$
 converts a string corresponding to a non-negative
 integer $n$ written in
  base $b$, to $n$ written in base
  $-b$ \cite[\S5.3]{AutomaticSequences}.

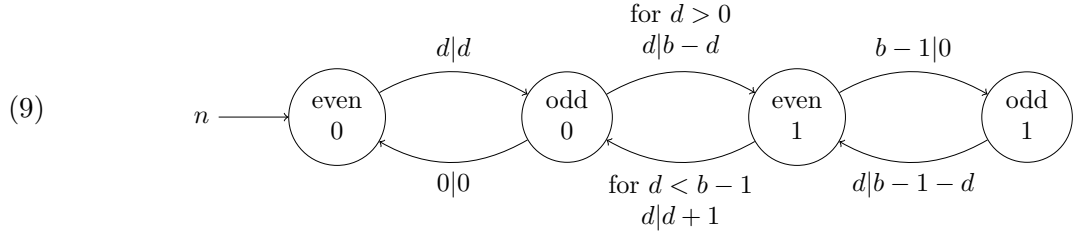
  \begin{equation}
 \label{eqn:convnto-b}   
\scalebox{0.9}{
  \begin{tikzpicture}[node distance=2 cm,baseline=(current  bounding  box.center)]
\pgfmathsetmacro{\minimalwidth}{0.1}
\node[circle,draw,inner sep=0pt,minimum size=0.1cm](a)
{$\begin{array}{c}\text{even}\\ 0\end{array}$};
\node[left of = a](s){$n$};
\node[circle,draw,inner sep=0pt,minimum size=0.1cm, right=of a](b)
{$\begin{array}{c}\text{odd}\\ 0\end{array}$};
\node[circle,draw,inner sep=0pt,minimum size=0.1cm, right=of b](c)
{$\begin{array}{c}\text{even}\\ 1\end{array}$};
\node[circle,draw,inner sep=0pt,minimum size=0.1cm, right=of c](d)
{$\begin{array}{c}\text{odd}\\ 1\end{array}$};
\draw[->] (a) to[out=30, in=150] node[above]{$d | d$} (b); 
\draw[->] (b) to[out=210, in=-30] node[below]{$0 | 0$}(a); 
\draw[->] (b) to[out=30, in=150] node[above]{$\begin{array}{c}\text{for }d>0\\
d | b-d\end{array}$} (c); 
\draw[->] (c) to[out=210, in=-30] node[below]{$\begin{array}{c}
\text{for }d<b-1\\
d | d+1\end{array}$} (b); 
\draw[->] (c) to[out=30, in=150] node[above]{$b-1 | 0$} (d); 
\draw[->] (d) to[out=210, in=-30] node[below]{$d | b-1-d$}(c); 
\draw[->] (s) -- (a); 
\end{tikzpicture}
}
  \end{equation}

  The input is a word in $\Sigma^*$ given by the
  base $b$ expansion of $n$, with
$d$ denoting one of the terms of the expansion.
The output is the base $-b$ expansion of $n$.
The states correspond to the parity of the digit position
in the expansion, and to the value
carried from one position to the next in computing the expansion.

A transducer with the
following diagram
converts from $n$ in base $b$ to $-n$ in base $-b$
\cite[\S5.3]{AutomaticSequences}.

  \begin{equation}
 \label{eqn:conv-nto-b}   
 \scalebox{0.9}{
\begin{tikzpicture}[node distance=2 cm,baseline=(current  bounding  box.center)]
\pgfmathsetmacro{\minimalwidth}{0.1}
\node[circle,draw,inner sep=0pt,minimum size=0.1cm](a)
{$\begin{array}{c}\text{odd}\\ 0\end{array}$};
\node[circle,draw,inner sep=0pt,minimum size=0.1cm, right=of a](b)
{$\begin{array}{c}\text{even}\\ 0\end{array}$};
\node[above of = b](s){$n$};
\node[circle,draw,inner sep=0pt,minimum size=0.1cm, right=of b](c)
{$\begin{array}{c}\text{odd}\\ 1\end{array}$};
\node[circle,draw,inner sep=0pt,minimum size=0.1cm, right=of c](d)
{$\begin{array}{c}\text{even}\\ 1\end{array}$};
\draw[->] (a) to[out=30, in=150] node[above]{$d | d$} (b); 
\draw[->] (b) to[out=210, in=-30] node[below]{$0 | 0$}(a); 

\draw[->] (b) to[out=30, in=150] node[above]{$\begin{array}{c}\text{for }d>0\\
d | b-d\end{array}$} (c); 
\draw[->] (c) to[out=210, in=-30] node[below]{$\begin{array}{c}
\text{for }d<b-1\\
d | d+1\end{array}$} (b); 
\draw[->] (c) to[out=30, in=150] node[above]{$b-1 | 0$} (d); 
\draw[->] (d) to[out=210, in=-30] node[below]{$d | b-1-d$}(c); 
\draw[->] (s) -- (b); 
\end{tikzpicture}
}
\end{equation}

\begin{lemma}
  \label{lem:1}
  The following transducer computes the
  function which takes the base $b$ expansion of $n$ to the
  base $b$ expansion of $n\oplus_{-b} -n$
  (as in Definition~\ref{def:ominus}).
  $$
  \scalebox{0.9}{
\begin{tikzpicture}
\pgfmathsetmacro{\minimalwidth}{0.1}
\node[circle,draw,inner sep=0pt,minimum size=1cm] at (0,0) (a){$00$};
\node[circle,draw,inner sep=0pt,minimum size=1cm] at (3,0)(b){$10$};
\node[circle,draw,inner sep=0pt,minimum size=1cm] at (6,0)(c){$11$};
\node(s) at (0,1.5){$n$};
\path[->,draw] (a) to  
[in=200,out=160,loop,distance=2cm] node[left] {$0|0$} (a);
\path[->,draw] (c) to  
[in=-20,out=20,loop,distance=2cm] node[right] 
{$b-1|0$} (c);
\path[->,draw] (b) to  
[in=-70,out=-110,loop,distance=2cm] node[below] {$\begin{array}{c}\text{for }0<d<b-1\\
d | 1\end{array}$} (b);
\draw[->] (a) to[out=30, in=150] node[above]{$\begin{array}{c}\text{for }d>0\\
d | 0\end{array}$} (b); 
\draw[->] (b) to[out=210, in=-30] node[below] {$0|1$}(a);
\draw[->] (b) to[out=30, in=150] node[above]{$b-1|1$} (c); 
\draw[->] (c) to[out=210, in=-30] node[below] {$\begin{array}{c}\text{for }d<b-1\\
d | 0\end{array}$}(b);
\draw[->](s)--(a);
\end{tikzpicture}
  }
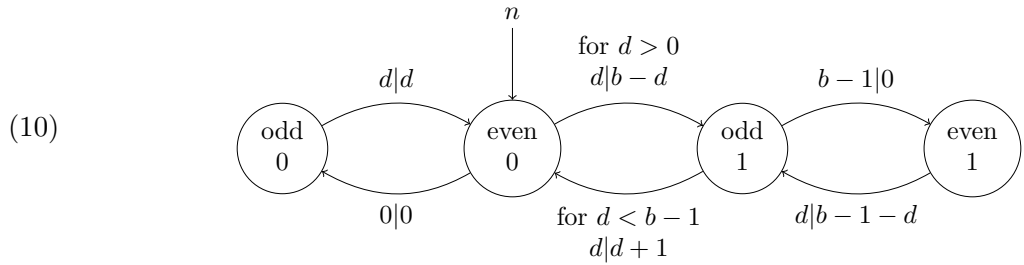
  $$
  
  \end{lemma}
\begin{proof}
In order to compute $n\oplus_{-b} -n$ we combine the output of the
two machines corresponding to diagrams (\ref{eqn:convnto-b})
  and (\ref{eqn:conv-nto-b}).
  To find the machine for
$n\oplus_{-b} -n$, we simultaneously compute both $n$ and $-n$ in base
  $-b$, digit by digit.  As each successive digit is computed, we immediately
  take the sum mod $b$.
  We move from odd to even states
simultaneously in computations of both $n$ and $-n$ in base $-b$, 
so the states of the combined transducer
depend on parity of the digit position,
and on whether the
digits
carried are
zero or not.  This gives a total of $8$ states; however, it turns out
that only $6$ states occur in a component of a state space containing
the initial state ``even 0,0'', as shown in Figure~\ref{fig}.
In the circles, the state is indicated by
listing the carry for obtaining the base $-b$ expansion of $n$, followed by 
the carry for the base $-b$ expansion of $-n$, below the digit position parity.
\begin{figure}
\scalebox{0.8}{   
\begin{tikzpicture}[scale={3.5}]
\pgfmathsetmacro{\minimalwidth}{0.1}
\node[circle,draw,inner sep=0pt,minimum size=0.1cm](a) at (0,0)
{$\begin{array}{c}\text{even}\\ 0,0\end{array}$};
\node[circle,draw,inner sep=0pt,minimum size=0.1cm](b) at (1,1)
{$\begin{array}{c}\text{odd}\\ 0,0\end{array}$};
\node at (-0.5,0)(s){$n$};
\node[circle,draw,inner sep=0pt,minimum size=0.1cm](c) at (3,1)
{$\begin{array}{c}\text{even}\\ 1,0\end{array}$};
\node[circle,draw,inner sep=0pt,minimum size=0.1cm](d) at (4,0)
{$\begin{array}{c}\text{odd}\\ 1,1\end{array}$};
\node[circle,draw,inner sep=0pt,minimum size=0.1cm](e) at (1,-1)
{$\begin{array}{c}\text{odd}\\ 0,1\end{array}$};
\node[circle,draw,inner sep=0pt,minimum size=0.1cm](f) at (3,-1)
{$\begin{array}{c}\text{even}\\ 1,1\end{array}$};
\draw[->] (a) to[out=60, in=215] node[above left]{$0 | 0\oplus 0=0$} (b); 
\draw[->] (b) to[out=240, in=30] node[below right]{$0 | 0\oplus 0=0$}(a); 

\draw[->] (a) to[out=-30, in=120] node{$\begin{array}{c}
\text{for }d>0\\d | d\oplus (b-d)=0\end{array}$} (e); 
\draw[->] (e) to[out=160, in=-70] node[below left]{$0 | 0\oplus 1=1$}(a); 

\draw[->] (b) to[out=30, in=150] node[above]{$\begin{array}{c}\text{for }d>0\\
d | (b-d)\oplus d=0 \end{array}$} (c); 
\draw[->] (c) to[out=210, in=-30] node[below]{$0|1\oplus 0=1$}(b);
\draw[->] (e) to[out=30, in=150] node[above]
{$b-1 | 1\oplus0=1$} (f); 
\draw[->] ([xshift=-4mm]e) --node[left]
{$\begin{array}{c}\text{for }0<d<b-1\\d | (b-d)\oplus(d+1)=1
\end{array}$} ([xshift=-4mm] c); 
\draw[->] (f) to[out=210, in=-30] node[above]{$\begin{array}{c}
\text{for }d<b-1\\
d | (d+1)\oplus(b-1-d)=0\end{array}$} (e); 

\draw[->]([xshift=4mm] c)--
node[xshift=19mm]{$\begin{array}{c}\text{for }0<d<b-1\\d | (d+1)\oplus(b-d)=1
\end{array}$}([xshift=4mm] e);
\draw[->] (c) to[out=-30, in=120] node[above]{$b-1|0\oplus 1=1$}(d);
\draw[->] (d) to[out=160, in=-70]
 node[xshift=-6mm]
{$\begin{array}{c}
\text{for }d<b-1\\d | (b-1-d)\oplus (d+1)=0\end{array}$} (c); 
\draw[->] (s) -- (a); 
\draw[->] (f) to[out=60, in=215] node[above left]{$b-1 | 0\oplus 0=0$} (d); 
\draw[->] (d) to[out=240, in=30] node[below right]{$b-1 | 0\oplus 0=0$}(f); 
\end{tikzpicture}
}
\caption{Diagram of a transducer with output $n\oplus_{-b}(-n)$.}
\label{fig}
\end{figure}
Outputs from the machines in
(\ref{eqn:convnto-b})
  and (\ref{eqn:conv-nto-b})
  are combined to give the xor output shown in
  Figure~\ref{fig}, where $x\oplus y$ means $x+y\mod b$.

Observing this diagram, we see that it can be collapsed into just three states,
since we see that in computing the output and transitions,
the parity is irrelevant; all that matters is what we are carrying,
and whether or not the carries in
transducers (\ref{eqn:convnto-b})
  and (\ref{eqn:conv-nto-b})
are the same or not.
The new equivalent diagram is as given in the statement of the Lemma.
\end{proof}

\begin{corollary}
The coefficients of the base $b$ expansion of $n\oplus_{-b}(-n)$
are all either $0$ or $1$.
  \end{corollary}
\begin{proof}
 We observe that all
possible outputs in Figure~\ref{fig} are $0$ or $1$.
  \end{proof}

\begin{theorem}
For all positive integers $b$ and $n$, we have
\begin{equation}
  \overline{\overline{((b+1)n)\ominus_b n}} = n\oplus_{-b} (-n),
  \label{equ:generalresultrestate}  
\end{equation}
with notation as in \S\ref{subsec:gen}.
\label{thm:main}
\end{theorem}
\begin{proof}
  First let's find a transducer
  $T_1=(Q_b:=\{0,\dots,b\},
  \Sigma_b:=\{0,\dots,b-1\},
  \delta_b,
  q_0=0,
  \Delta_b=\Sigma_b,
  \lambda_b
  )$,
  and corresponding 
  diagram,
  for multiplication by $b+1$ in base $b$.
  The input and output are both elements of $\Sigma_b^*$,
  corresponding to base $b$ expansions of integers.
  
When we multiply a number $n=a_kb^k + a_{k-1}b^{k-1}+\cdots + a_0$ by
$b+1$, in order to obtain an expansion
$(b+1)n=c_mb^m + c_{m-1}b^{m-1}+\cdots + c_0$, we apply the usual
multiply and carry, with the ``carry'' value, denoted
by $c$ in the following, corresponding to the state.

E.g., suppose for some $n_1<k$ we have
$$(b+1)(a_{n_1}b^{n_1+1}+\cdots + a_0)=
cb^{n_1+1} + (c_{n_1}b^{n_1}+\cdots + c_0),$$
where $c_0,\dots,c_{n_1}$ are integers between
$0$ and $b-1$, and $c$ is a non-negative integer, then
$$(b+1)(a_{n_1+1}b^{n_1+1} + a_{n_1}b^{n_1+1}+\cdots + a_0)=
c'b^{n_1+2} + c_{n_1+1}b^{n_1+1} + (c_{n_1}b^{n_1}+\cdots + c_0),$$
where the output value $\lambda_b(c_{n_1},a_{n_1+1})$ is
$$c_{n_1+1}=(a_{n_1+1} + c_{n_1})\mod b,$$
and the new state, $\delta_b(c_{n_1},a_{n_1+1})$, is $$c'=
\left\lfloor\frac{a_{n_1+1} + c_{n_1}}{b}\right\rfloor+a_{n_1+1}.$$

As part of a transition diagram, removing subscripts, this appears as
the following figure.
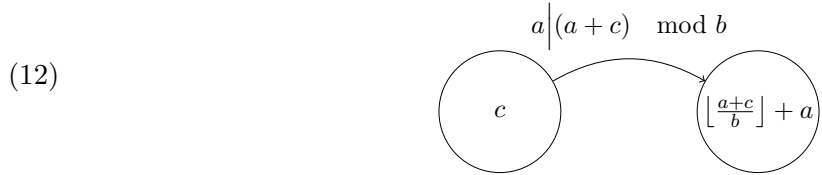
\begin{equation}
  \label{eqn:t1}
  \scalebox{0.9}{
  \begin{tikzpicture}[node distance=2 cm,baseline=(current  bounding  box.center)]
\pgfmathsetmacro{\minimalwidth}{0.1}
\node[circle,draw,inner sep=0pt,minimum size=1.8cm](a)
{$c$};
\node[circle,draw,inner sep=0pt,minimum size=1.8cm, right=of a](b)
{$
\left\lfloor\frac{a + c}{b}\right\rfloor+a$};
\draw[->] (a) to[out=30, in=150] node[above]{$
a \Bigl\lvert(a + c)\mod b$} (b); 
  \end{tikzpicture}
  }
\end{equation}
Now if we combine the multiplication by $b+1$ with a mod $b$ subtraction,
corresponding to the operation
$n\mapsto \bigl((b+1)n\bigr)\ominus_b n$, since we just need
to subtract the  $k^{th}$ component 
of $n$ from
the $k^{th}$ component of $(b+1)n$, modulo $b$, in the base $b$ expansions,
this becomes the following figure.
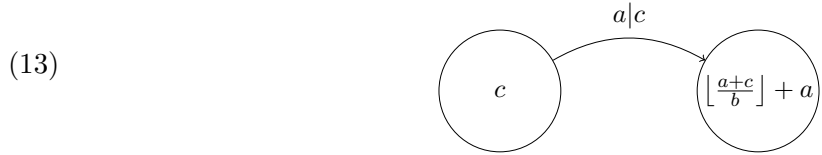
\begin{equation}
  \label{eqn:t2}
  \scalebox{0.9}{
\begin{tikzpicture}[node distance=2 cm,baseline=(current  bounding  box.center)]
\pgfmathsetmacro{\minimalwidth}{0.1}
\node[circle,draw,inner sep=0pt,minimum size=1.8cm](a)
{$c$};
\node[circle,draw,inner sep=0pt,minimum size=1.8cm, right=of a](b)
{$
\left\lfloor\frac{a + c}{b}\right\rfloor+a$};
\draw[->] (a) to[out=30, in=150] node[above]{$
a | c$} (b); 
\end{tikzpicture}
}
\end{equation}

In a diagram for a machine which takes the input being the base $b$
expansion for $n$, and the output the base $b$ expansion for
$(b+1)n$, the states correspond to the possible carry values,
which are $0,1,\dots,b$.
To see this, note that if $0\le c\le b$ and
$0\le a< b$ then $0\le a+c<2b$, so
the state label  (the carry value, shown as the target
state in (\ref{eqn:t1}) and (\ref{eqn:t2})) satisfies
$\left\lfloor\frac{a + c}{b}\right\rfloor<2$
so 
$\left\lfloor\frac{a + c}{b}\right\rfloor +a< 1 + a<1+b\le b$.
To summarize, the transducer
corresponding to
the operation
$n\mapsto ((b+1)n)\ominus_b n$, with input and output 
written in base $b$, satisfies the following
properties.
\begin{itemize}
\item
  From the state $0$, we pass to state zero only if
$a=0$, since
$\left\lfloor\frac{a + c}{b}\right\rfloor +a\ge a$.
\item
  From the state $0$, with input $0<a$, we pass 
to state $a$; note that all inputs $a$ satisfy $a<b$.
\item
  From state $c$ with $0<c<b$,
\begin{itemize}
\item
  we can pass to state $0$ precisely when
the input $a$ is zero.
\item
we can pass to state $b$ precisely whenever $a=b-1$, since $a+c\ge b$.
\item
For all other inputs, we pass to another state $c'$ with
$0<c<b$.
\end{itemize}
\item
For state $c=b$, in the case
$a=b-1$, we remain in state $b$.  Otherwise we pass to a state
with value between $1$ and $b-1$.
\end{itemize}
Finally, we want a diagram for a machine corresponding to the
operation
$$n\mapsto\overline{\overline{\bigl((b+1)n\bigr)\ominus_b n}}.$$
The reduction $\overline{\overline{\alpha}}$ replaces
a zero term in the base $b$ expansion sequence for $\alpha$
by zero, and all other terms by $1$.
By considering the transitions shown
in (\ref{eqn:t2}), we see that
the output $c$ of the transducer
for $n\mapsto\bigl((b+1)n\bigr)\ominus_b n$
will only be zero modulo $b$ when 
$c$ is divisible by $b$, which happens only when  $c=0$
or $c=b$.
Thus, when we 
consider a machine for 
$n\mapsto\overline{\overline{\bigl((b+1)n\bigr)\ominus_b n}}$,
where we reduce all outputs 
of $n\mapsto\bigl((b+1)n\bigr)\ominus_b n$,
to $0$ or $1$,
our machine has the following diagram:
$$
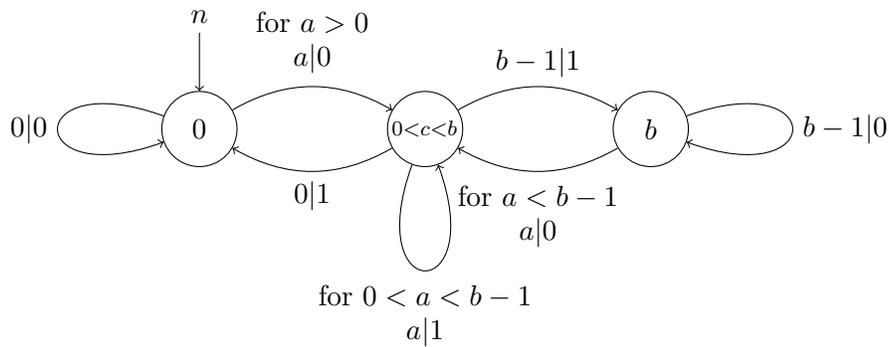
\begin{tikzpicture}
\pgfmathsetmacro{\minimalwidth}{0.1}
\node[circle,draw,inner sep=0pt,minimum size=1cm] at (0,0) (a){$0$};
\node[circle,draw,inner sep=0pt,minimum size=1cm] at (3,0)(b){$_{0<c<b}$};
\node[circle,draw,inner sep=0pt,minimum size=1cm] at (6,0)(c){$b$};
\node(s) at (0,1.5){$n$};
\path[->,draw] (a) to  
[in=200,out=160,loop,distance=2cm] node[left] {$0|0$} (a);
\path[->,draw] (c) to  
[in=-20,out=20,loop,distance=2cm] node[right] 
{$b-1|0$} (c);

\path[->,draw] (b) to  
[in=-70,out=-110,loop,distance=2cm] node[below] {$\begin{array}{c}\text{for }0<a<b-1\\
a | 1\end{array}$} (b);

\draw[->] (a) to[out=30, in=150] node[above]{$\begin{array}{c}\text{for }a>0\\
a | 0\end{array}$} (b); 
\draw[->] (b) to[out=210, in=-30] node[below] {$0|1$}(a);

\draw[->] (b) to[out=30, in=150] node[above]{$b-1|1$} (c); 
\draw[->] (c) to[out=210, in=-30] node[below] {$\begin{array}{c}\text{for }a<b-1\\
a | 0\end{array}$}(b);

\draw[->](s)--(a);
\end{tikzpicture}
$$
We see by inspection that, up to relabeling, this is exactly the same machine
as in
Lemma~\ref{lem:1}, hence proving the result.
\end{proof}

\section{Conclusion}

It was not immediately obvious how
(\ref{equ:originalconjecture}) should be
generalized.  Using transducers not only gives an elegant algorithmic proof,
but also helped in the discovery of the generalization given in
Theorem~\ref{thm:main}.
  A transducer encapsulates a way to compute a function,
  which simplifies proofs, since we have a clearly laid out visualization
  of the construction of the functions in question.

\end{document}